\documentclass[12pt,letterpaper]{amsart}
\usepackage{geometry}
\usepackage{mathpazo} 
\usepackage{amsmath,amssymb,amsfonts,amsthm,amscd}
\usepackage{mathrsfs,latexsym,url,setspace,xcolor,graphicx}
\newtheorem{theorem}{Theorem}
\newtheorem{theoremA}{Theorem}
\newtheorem*{theorem*}{Theorem}

\newtheorem{lemma}{Lemma}
\newtheorem{corollary}{Corollary}
\theoremstyle{remark}

\newtheorem{remark}{Remark}

\newcommand{\spec}{\text{Spec}}
\newcommand{\ad}{\text{ad}}

\newcommand{\gr}{\text{gr}}
\newcommand{\der}{\text{Der}}

\usepackage[author={S},color=yellow,opacity=0.6,markup=Highlight]{
pdfcomment}
\onehalfspace

\title{Derivations of quantizations in characteristic $p$}

\author{Akaki Tikaradze}
\email{ tikar06@gmail.com}
\address{University of Toledo, Department of Mathematics \& Statistics, 
Toledo, OH 43606, USA}

\date{\today}
\begin{document}

\begin{abstract}
 Let $\bf{k}$ be an algebraically closed field of odd characteristic.  We describe derivations of 
 a large class of quantizations of affine normal Poisson
 varieties over $\bf{k}.$
\end{abstract}
\maketitle


Let $\bf{k}$ be an algebraically closed field of characteristic $p>2.$ 
Let $A$ be an associative $\bf{k}$-algebra and let
$Z$ be its center. Then we have the natural restriction map 
$HH^1(A)\to \der_{\bold{k}}(Z, Z)$ from the first Hochschild cohomology of $A$ over $\bf{k}$
to $\bf{k}$-derivations of $Z.$ In this note we show that this map is injective for
a large class of quantizations of Poisson algebras (Theorem \ref{main}) and is an isomorphism
for central quotients of the enveloping algebras of semi-simple Lie algebras (Corollary \ref{qow}).
 It is well-know that this map is an
isomorphism if $A$ is an Azumaya algebra over $Z.$ In fact in this case all corresponding Hochschild cohomology
groups are isomorphic $HH^*(A)\cong HH^*(Z).$

Throughout given an element $x\in A,$ by $\ad(x)$ we will denote the commutator bracket \\$[x, -]:A\to A$ as it
is customary.  Thus we have an injective homomorphism of $Z$-modules $A/Z\overset{\ad}\rightarrow \der_{\bold{k}}(A, A).$
We have a short exact sequence of $Z$-modules $$0\to A/Z\overset{\ad}\rightarrow \der_{\bf{k}}(A, A)\to HH^1(A)\to 0.$$
We will be interested in determining whether this sequence splits. We will start by recalling how a deformation of
$A$ over $W_2(\bf{k})$ gives rise to a Poisson bracket on $Z,$ where $W_2(\bf{k})$ denotes the ring of Witt vectors of length $2$ over
$\bf{k}.$ Hence $W_2(\bf{k})$ is a free $\mathbb{Z}/p^2\mathbb{Z}$-module \\and $W_2(\bold{k})/pW_2(\bold{k})=\bold{k}.$

  Let $A_2$ be a lift of $A$ over $W_2(\bf{k}).$ Thus $A_2$ is an associative $W_2(\bf{k})$-algebra which is free as a $W_2(\bf{k})$-module and $A_2/pA_2=A.$ Then we have a derivation  $i:Z\to HH^1(A)$ defined as follows. For $z\in Z$, let $\tilde{z}\in A_2$ be a lift of $ z.$ Then 
$\ad(\tilde{z})(A_2)\subset pA_2.$ Hence $$i(z)=(\frac{1}{p}\ad(\tilde{z}))\mod  p: A\to A$$ is a derivation 
which is independent of a lift of $z \mod$ inner derivations. The map $i$ restricted on $Z$ gives rise to 
the Poisson bracket $\lbrace  , \rbrace :Z\times Z\to Z,$
which we will refer to as the deformation Poisson bracket on $Z.$

Then we have the following

\newpage

\begin{lemma}\label{L1}
Let $A$ be an associative $\bf{k}$-algebra, and let $A_2$ be its lift over $W_2(\bf{k})$. 
Let $Z$ be the center of $A.$ Assume that $Z$ admits a lift as a subalgebra of $A_2.$  Assume that 
$\spec Z$ is a normal variety such that the deformation Poisson bracket on $Z$ is
symplectic on the smooth locus of $\spec Z$, and $A$ is a finitely generated Cohen-Macaulay $Z$-module
such that $Ann_ZA=0.$
 Then the restriction map $\der_{\bold{k}}(A, A)\to \der_{\bold{k}}(Z, Z)$ admits a $Z$-module splitting.

\end{lemma}

\begin{proof}

Let $\tilde{Z}\subset \tilde{A}$ be an algebra lift of $Z$ over $W_2(\bold{k}).$
 Thus $\tilde{Z}$ is a subalgebra of $\tilde{A}$ free over $W_2(\bold{k})$ such that $\tilde{Z}/p\tilde{Z}=Z.$ 
Then we have a map $(\frac{1}{p}\ad)\mod p:\tilde{Z}\to \der_{\bf{k}}(A, A).$ This map clearly factors through a derivation
$\tilde{Z}/p\tilde{Z}=Z\to \der_{\bf{k}}(A, A)$ and is a lift of the map $i:Z\to HH^1(A)$ described above. 
Let $U$ be the smooth locus of $\spec Z$. Thus by the assumption the deformation Poisson bracket of $\spec Z$ is symplectic on $U$.
We have the map of coherent sheaves $\tilde{i}|_U:\Omega^1_{U}\to \der_{\bold{k}}(A,A)|_{U},$
and composing it with the identification by the symplectic form between tangent and cotangent bundles $T^1_U\to \Omega^1_U,$ we get
a map of coherent sheaves on $U,$ $\tau:T_U\to \der(A, A)|_U.$
 Since $codim(\spec Z\setminus U)\geq 2$ and $\spec Z$ is a normal variety, then 
$\Gamma(U, T_{U})=\der_{\bf{k}}(Z, Z).$ Also $\Gamma(U, A_{U})=A$ since $A$ is a Cohen-Macaulay $Z$-module
of dimension $\dim Z.$ Thus we get a map of $Z$-modules $\tau:\der_{\bold{k}}(Z, Z)\to \der_{\bold{k}}(A, A)$
 which is a section of the restriction map $\der_{\bold{k}}(A, A)\to \der_{\bold{k}}(Z, Z).$

\end{proof}








Next we have the following

\begin{lemma}\label{L2}
Let $A$ be an associative $\bf{k}$-algebra which is a finite over its center $Z$.
Assume that $Z$ is a normal $\bf{k}$-domain such that $A/Z$ is a Cohen-Macaulay module over $Z,$
and $Ann_Z(A/Z)=0.$ 
 Assume moreover that the Azumaya locus of $A$ has a compliment of codimension $\geq 2$ in $\spec Z.$
Then the restriction map $HH^{1}(A)\to \der_{\bold{k}}(Z, Z)$ is injective.
\end{lemma}
\begin{proof}

Let $D:A\to A$ be a $\bf{k}$-derivation
such that $D(Z)=0.$ Let $U$ be the Azumaya locus of $A.$ Put $Y=\spec Z\setminus U.$ Since $A|_{U}$ is Azumaya algebra, it follows that 
there exists $x\in \Gamma (U, A/Z)$ such that $D|_{U}$ is equal to
$\ad(x).$ Since $A/Z$ is Cohen-Macaulay module over $Z$ and $Ann_Z(A/Z)=0,$ it follows that
$depth_Y(A/Z)\geq 2.$ Thus the standard argument using local cohomology
groups \cite{H} implies that $\Gamma(U, A/Z)=A/Z.$ It follows that there exists
$x\in A$ such that $D-\ad(x)x$ vanishes on $U.$ Since $Z$ is normal, it follows that $depth_YA\geq 2$. Hence
$\Gamma(U, A)=A.$ Therefore $D-\ad(x)=0,$ hence $D=\ad(x)$ is an inner derivation.
We conclude that $HH^1(A)\to \der_{\bold{k}}(Z, Z)$ is injective as desired. 
\end{proof}







Let  an associative $\bf{k}$-algebra $A$ be equipped with an algebra filtration
 $1\in A_0\subset A_1\subset \cdots$
such that the associated graded algebra $\gr A=\bigoplus_n A_n/A_{n-1}$ is commutative. Then recall that there is
a graded Poisson bracket on $\gr A$ defined as follows. Given $x\in A_{n}/A_{n-1}, y\in A_{m}/A_{m-1}$, then
their Poisson bracket $\lbrace x, y\rbrace$ is defined to be $[\tilde{x}, \tilde{y}]\in A_{n+m-1}/A_{n+m-2},$
 where $\tilde{x}\in A_n,\tilde{y}\in A_m$
are lifts of $x, y.$ In this setting we say that a filtered algebra $A$ as a quantization of a
graded Poisson algebra $grA.$
This is closely related to deformation quantizations: By taking $\tilde{A}$ to be ($h$-completion of)
the Rees algebra of $A: R(A)=\bigoplus_{n} A_n\otimes h^n$, then $\tilde{A}/h\tilde{A}=\gr A.$

We will need the following computation which relates the deformation Poisson bracket on $Z$ to
the Poisson bracket on $\gr A.$ This computation is similar and motivated by a result of
Kanel-Belov and Kontsevich \cite{KK},
where the the Poisson bracket on $Z$ was computed when $A$ is the Weyl algebra.

\begin{lemma}\label{bracket}

Let $A$ be a filtered $W_2(\bf{k})$-algebra, such that $\gr A=B$ is commutative and free over $W_2(\bold{k})$.  Put
$\overline{A}=A/pA, \overline{B}=B/pB.$ Let $\overline{Z}$ denote the center of  $\overline{A}.$ Assume
that $gr(\overline{Z})=\overline{B}^p.$ Then the top degree part of the deformation Poisson bracket on $\overline{Z}$
is equal to -1 times the Poisson bracket of $\overline{B}.$

\end{lemma}
\begin{proof}

	We will verify that given central elements $\bar{x}, \bar{y}\in \bar{Z}$ such that
	 $\gr (\bar{x})=\bar{a}^p, \gr (\bar{y})=\bar{b}^p, \bar{a}, \bar{b}\in \overline{B},$ then $\gr([x, y])=p\lbrace \bar{a}, \bar{b}\rbrace.$
Here $x, y \in A$ are lifts of $\bar{x}, \bar{y}$ respectively.

It will be more convenient to work in the deformation quantization setting. Thus we will assume that $A=B[[h]]$ as a a free $W_2[k][[h]]$-module
such that  $$A/hA=B, [a, b]=h\lbrace a, b\rbrace \mod h^2,  a, b\in B.$$
Then by our assumption $\overline{Z}=\lbrace \bar{a}^p-h^{p-1}\bar{a}_{[p]}, \bar{a}\in \overline{B}\rbrace.$ 
Thus $\ad_P(\bar{a})^p=ad_P(\bar{a}_{[p]})$, here $\ad_P(x)$ denotes the Poisson bracket $\lbrace x, -\rbrace, x\in B.$
We will compute the Poisson bracket on $\overline{Z} \mod h^{p+1}$. Thus
without loss of
generality we will put $h^{p+1}=0.$ Let $x=a^p-h^{p-1}a_{[p]}, y=b^p-h^{p-1}b_{[p]}.$ We want to compute $[x, y].$ We have
$$[a^p, y]=\ad(a)^p(y)-\sum_{i=1}^{p-1}(-1)^i{p\choose i }a^iya^{p-i}.$$ Since ${p\choose i}y$ is in the center of $A$, we have that
$$\sum_{i=1}^{p-1}(-1)^i{p\choose i}a^iya^{p-i}=\left(\sum_{i=1}^{p-1}(-1)^i{p\choose i}\right)a^py=0.$$
So $[a^p, y]=\ad(a)^p(y).$ We have $$\ad(a)^p(y)=\ad(a)^p(b^p)-h^{p-1}\ad(a)^p(b_{[p]}).$$ But $\ad(a)^p(a)\subset h^pA$
thus $h^{p-1}\ad(a)(A)=0.$ So $\ad(a)^p(y)=\ad(a)^p(b^p).$ On the other hand $$[h^{p-1}a_{[p]}, y]=h^{p-1}[a_{[p]}, b^p].$$
Now since 
$$\ad(a)^p(b^p)=ph^{p}\ad_P(a)(b^{p-1}\lbrace a, b\rbrace),\quad h^{p-1}[a_{[p]}, b^p]=ph^{p}\lbrace a_{[p]}, b\rbrace b^{p-1},$$ 
we obtain that 
$$[x, y]=ph^{p}\left(\ad_P(a)^{p-1}(b^{p-1}\lbrace a, b\rbrace)-\lbrace a_{[p]}, b\rbrace b^{p-1}\right)=(p-1)!\lbrace a, b\rbrace^p.$$

Here we used the following identity. Let $D:B\to B$ be a derivation, then $$D^{p-1}(b^{p-1}D(b))=(p-1)!D(b)^p+b^{p-1}D^P(b), b\in B.$$

\end{proof}

Recall that a reduced commutative $\bf{k}$-algebra $B$ is said to be Frobenius split if the quotient map
$B\to B/B^p$ splits as a $B^p$-module homomorphism.

\begin{theorem}\label{main}
Let  $\spec B$ be a normal Frobenius split Cohen-Macaulay Poisson variety over
$\bf{k}$ such that the Poisson bracket on the smooth locus of $\spec B$ is symplectic.
Let $A$ be a  quantization of $B$ such that $\gr Z=B^p,$ where $Z$ is the center of $A.$ Moreover, assume that
$A$ admits a lift to $W_2(\bf{k}).$ Let $U$ denote the smooth locus of $\spec Z.$
Then the restriction map $HH^1_{\bold{k}}(A)\to\der_{\bold{k}}(Z,Z)$ is injective and its cokernel
is a quotient of $\Omega^1(U)/\Omega^1_Z$ as a $Z$-module.

\end{theorem}

\begin{proof}
It was shown in \cite{T} that the Azumaya locus of $A$ in $\spec Z$ has the compliment of codimension $\geq 2.$
Normality of $B$ implies that $Ann_{B^p}(B/B^p)=0.$
Since $B/B^p$ is a direct summand of $B$, and $B$ is a Cohen-Macaulay $B^p$-module, it follows that $B/B^P$
is a Cohen-Macaulay $B^p$-module of dimension $dim B^p.$ 
Since $\gr(A/Z)=B/B^p$ and $\gr Z=B^p,$
it follows that $ A/Z$ is a Cohen-Macaulay $Z$-module and $Ann_Z(A/Z)=0$.  Now Lemma \ref{bracket} implies that 
the Poisson bracket on $Z$ coming from a lift of  $ A$ over  $ W_2(\bf{k})$ is a deformation
of the Poisson bracket on $B,$ hence it is symplectic on an open subset of $\spec Z$ whose compliment 
has codimension $\geq 2.$ 
Thus all assumptions of Lemma \ref{L2} are satisfied. 

 Denote by $P$ the $Z$-span of derivations of the form $a\lbrace b, -\rbrace, a, b\in Z.$ Clearly $P$ is in the image
 of the restriction $HH^1_{\bold{k}}(A)\to\der_{\bold{k}}(Z,Z).$
Then we have a $Z$-module map $\Omega_Z\to P\subset \der_{\bold{k}}(Z, Z)$ corresponding to the Poisson bracket,
 and $\der_{\bold{k}}(Z, Z)$ can be identified with
$\Gamma( U, \Omega)$ via the symplectic pairing. Hence  $\Omega^1(U)/\Omega^1_Z$ maps onto the cokernel
of the restriction map $HH^1_{\bold{k}}(A)\to\der_{\bold{k}}(Z,Z)$

\end{proof}

This result applies to a large class of algebras including symplectic reflection algebra. Our next result shows that
the restriction map from Theorem \ref{main} is an isomorphism for the case of  central quotients of enveloping algebras of semi-simple Lie algebras. 
Let us recall their definition and fix
the appropriate notations first.


Let $\mathfrak{g}$ be a Lie algebra of a connencted semi-simple simply connected algebraic group $G$ over $\bf{k}$, 
assume that $p$ is large enough relative to $\mathfrak{g}$
(for example $p$ is very good for $G$.)
Let $Z_0\subset Z(\mathfrak{U}\mathfrak{g})$ denote $G$-invariants of  the enveloping algebra $\mathfrak{U}\mathfrak{g}$ 
under the adjoint action of $G$.
Let $\chi:Z_0\to \bf{k}$ be a character. 
Put $\mathfrak{U}_{\chi}\mathfrak{g}=\mathfrak{U}\mathfrak{g}/Ker(\chi)\mathfrak{U}\mathfrak{g}.$
\newpage
\begin{corollary}\label{qow}
Let $A$ be a quotient
enveloping algebra $\mathfrak{U}_{\chi}\mathfrak{g}.$ Let $Z$ be the center of $A.$
Then we have an isomorphism of $Z$-modules $\der_{\bold{k}}(A, A)\cong A/Z\bigoplus \der_{\bold{k}}(Z, Z).$

\end{corollary}
\begin{proof}
Let $\tilde{g}$ be a Lie algebra lift of $\mathfrak{g}$ over $W_2(\bold{k})$.
Let $t_1,\cdots, t_n$ be generators of $ker(\chi)$ and 
$\tilde{t_1},\cdots,\tilde{t_n}$ be their lift in $Z(\mathfrak{U}\mathfrak{\tilde{g}}).$
Thus $Z(\mathfrak{U}\mathfrak{g})$ is generated by $t_1,\cdots, t_n$ \\over $Z_p=\lbrace g^p-g^{[p]}, g\in \mathfrak{g}\rbrace$.
Also, $Z$ is the quotient of $Z_p.$
Let $A_2= \mathfrak{U}\mathfrak{\tilde{g}}/(\tilde{t_1},\cdots, \tilde{t_n}).$ So $A_2$ is a lift of $A$ over $W_2(\bf{k}). $
We will show that $Z$ admits an algebra lift in $A_2.$
 Let $[p]:\tilde{\mathfrak{g}}\to \tilde{\mathfrak{g}}$ be a lift of
the restricted structure map $[p]:\mathfrak{g}\to\mathfrak{g}.$ Then if follows from computation in Lemma \ref{bracket} that
$$[x^p-x^{[p]}, y^{p}-y^{[p]}]=-p([x, y]^p-[x, y]^{[p]})\quad  x, y\in \tilde{\mathfrak{g}}.$$ 
Let $\tilde{\mathfrak{g}}_1$  be a $W_2(\bold{k})$-Lie algebra
such that  $\tilde{\mathfrak{g}}_1=\mathfrak{g}$ as $W_2(\bold{k})$-module and the Lie bracket $[x, y]_{\tilde{\mathfrak{g}}_1}$ 
is defined as $-p[x, y]_{\mathfrak{\tilde{g}}}.$
Thus we have an algebra map 
$i:\mathfrak{U}\mathfrak{\tilde{g_1}}\to \mathfrak{U}\mathfrak{\tilde{g}}$ where $i(x)=x^p-x^{[p]}, x\in \mathfrak{\tilde{g}}.$
Denote the image of $i$ by $S.$ Thus $S$ is an algebra lift of $Z_p$ in $\mathfrak{U}\mathfrak{\tilde{g}}.$
 Let $S_1$ denote the image of $S$ under the quotient map
$\mathfrak{U}\mathfrak{\tilde{g}}\to \mathfrak{U}\mathfrak{\tilde{g}}/(\tilde{t_1},\cdots, \tilde{t_n})=A_2.$ Therefore $S'$ is an algebra lift of
$Z$ in $A_2.$

Using  the usual PBW filtration of $A$ we have
$\gr (A)=k[N],$ where $N$ is the nilpotent cone of $\mathfrak{g}.$ Now since $N$ is a Frobenius
split normal Cohen-Macaulay variety \cite{BK}, and the Poisson bracket on the regular locus of $N$ is symplectic, 
Theorem \ref{main} and  Lemma \ref{L1} imply the desired result.

\end{proof}












\begin{remark}
It is known that in characteristic 0 Hochschild cohomology of symplectic reflection algebras $H$ is concentrated in even dimensions \cite{GK},
so it has no outer derivations. The same is true for the enveloping algebras $\mathfrak{U}\mathfrak{g}$ and its quotients.

 \end{remark}

 \noindent\textbf{Acknowledgement:} I am very grateful to Marton Hablicsek for pointing out a serious mistake in the earlier version of the paper.
 
 

\begin{thebibliography}{qowq}
\bibitem [KK]{KK}
A.~Belov-Kanel,  M.~ Kontsevich, {\em Automorphisms of the Weyl algebra}, Lett. Math. Phys. 74
(2005), 181--199.



\bibitem [BFG]{BFG}
 R.~Bezrukavnikov, M.~Finkelberg, V.~Ginzburg, {\em Cherednik algebras and Hilbert schemes in characteristic p}, With an Appendix
by P.~Etingof,  Represent. Theory 10 (2006), 254--298.


\bibitem [BK]{BK}
M.~Brion, S.~Kumar, {\em Frobenius splitting methods in geometry and representation theory}, Progress in Mathematics, 231.
Birkhauser Boston, Inc., Boston, MA, 2005.

\bibitem [EG]{EG}
P.~Etingof, V.~Ginzburg, {\em Symplectic reflection algebras, Calogero-Moser space, and deformed Harish-Chandra homomorphism},
 Invent. Math. 147 (2002), no. 2, 243--348. 

\bibitem [GK]{GK}

V.~Ginzburg, D.~Kaledin, {\em Poisson deformations of symplectic quotient singularities}, adv. Math. 186 (2004), 1--57




\bibitem [H]{H}

R.~Hartshorne, {\em Local cohomology}, Lecture notes in Mathematics, volume 41, 1967 




\bibitem [T]{T}
A.~Tikaradze, {\em On the Azumaya locus of almost commutative algebras}, Proc. Amer. Math. Soc. 139 (2011), no. 6, 1955--1960.

\end{thebibliography}

\end{document}